\documentclass[a4paper,11pt]{amsart}
\usepackage{amsmath}
\usepackage{amsbsy} 
\usepackage{amsthm}
\usepackage{amssymb}
\usepackage{amscd}
\usepackage{fullpage}
\usepackage{lmodern}
\usepackage[utf8]{inputenc}
\usepackage[all]{xy}
\theoremstyle{break}
\newtheorem{theorem}{Theorem}
\newtheorem{proposition}[theorem]{Proposition}
\newtheorem{cororally}[theorem]{Corollary}
\newtheorem{lemma}[theorem]{Lemma}

\theoremstyle{definition}
\newtheorem{definition}[theorem]{Definition}
\newtheorem{example}[theorem]{Example}
\newtheorem{remark}[theorem]{Remark}

\newcommand{\C}{\mathbb{C}}
\newcommand{\R}{\mathbb{R}}
\newcommand{\N}{\mathbb{N}}
\newcommand{\Z}{\mathbb{Z}}
\newcommand{\Q}{\mathbb{Q}}

\newcommand{\D}{\mathcal{D}}

\newcommand{\F}{\mathbb{F}}

\newcommand{\supp}{{\rm supp}}
\newcommand{\Fix}{{\rm Fix}}
\renewcommand{\d}{\partial}

\renewcommand{\i}{{\bf i}}

\newcommand{\e}{\varepsilon}
\newcommand{\Cs}{\mbox{${\rm C}^\ast$}}

\author{Ryo Ochi}
\title{A remark on the freeness condition of Suzuki's correspondence theorem for intermediate \Cs-algebras}

\begin{document}
\maketitle
\begin{abstract}
  Let $\Gamma$ be a discrete group satisfying the approximation property (AP).
  Let $X$, $Y$ be $\Gamma$-spaces and $\pi \colon Y \to X$ be a proper factor map
  which is injective on the non-free part. 
  We prove the one-to-one correspondence between intermediate \Cs-algebras of $C_0(X) \rtimes_r \Gamma \subset C_0(Y) \rtimes \Gamma$
  and intermediate $\Gamma$-\Cs-algebras of $C_0(X) \subset C_0(Y)$.
  This is a generalization of Suzuki's theorem that proves the statement for free actions.
\end{abstract}

\section{Introduction}
Let $\Gamma$ be a discrete group and $X$, $Y$ be locally compact spaces on which $\Gamma$ acts.
Let $\pi \colon Y \to X$ be a proper factor map.
We study the relation between intermediate $\Gamma$-\Cs-algebras of $C_0(X) \subset C_0(Y)$
and intermediate \Cs-algebras of $C_0(X) \rtimes_r \Gamma \subset C_0(Y) \rtimes_r \Gamma$.

Inclusions of operator algebras play an important role in many subjects including operator theory and knot theory. 
Structures of subalgebras of \Cs-algebras have been studied by many hands
(\cite{izumi1998galois}, \cite{izumi2002inclusions}, \cite{ge1996tensor}, \cite{zacharias2001splitting}, \cite{zsido2000criterion}, etc.).

A Galois correspondence theorem in operator algebras refers to a type of structure results for subalgebras of crossed products and fixed point subalgebras of operator algebras.
This is proved in many cases.
More precisely, a Galois correspondence is
that for an operator algebra $M$ on which a compact group $G$ (or a discrete group $\Gamma$) acts, 
there exists a one-to-one correspondence between intermediate operator algebras of $M^G \subset M$ and closed subgroups of $G$
(or a one-to-one correspondence between intermediate operator algebras of $M \subset M \rtimes \Gamma$ and subgroups of $\Gamma$).
Izumi, Longo and Popa \cite{izumi1998galois} prove the Galois correspondence
for a factor $M$ on which a compact group $G$ acts minimally (or a discrete group $\Gamma$ acts outerly).
In \cite{izumi2002inclusions}, Izumi proves the Galois correspondence for a simple $\sigma$-unital \Cs-algebra on which a finite group acts outerly.

Ge and Kadison \cite{ge1996tensor} prove the tensor splitting theorem
that for every factor $M$ and every von Neumann algebra $N$,
every von Neumann subalgebra of $M \overline{\otimes} N$ which contains $M$ is of the form $M \overline{\otimes} N_0$
for some von Neumann subalgebra $N_0$ of $N$.
In the case of simple \Cs-algebras, the tensor splitting theorem is established under some conditions (see \cite{zacharias2001splitting}, \cite{zsido2000criterion}).

Suzuki proves the following theorem among others in \cite{suzuki2018complete}.
\begin{theorem}[Suzuki, \cite{suzuki2018complete}, Main Theorem ($\Cs$-case)]
  Let $\Gamma$ be a discrete group satisfying the AP.
  Let $X$, $Y$ be $\Gamma$-spaces on which $\Gamma$ acts freely and $\pi$ be a proper factor map from $Y$ to $X$.
  Then the map
  \begin{align*}
    C_0(Z) \mapsto C_0(Z) \rtimes_r \Gamma
  \end{align*}
  gives a lattice isomorphism between the lattice of intermediate extensions of $\pi$ and that of intermediate \Cs-algebras of $C_0(X) \rtimes_r \Gamma \subset C_0(Y) \rtimes_r \Gamma$.
\end{theorem}

The freeness condition cannot be removed in general (see \cite[Proposition 2.6]{suzuki2018complete}).
The following theorem generalizes the above result of Suzuki by relaxing the freeness condition. 
\begin{theorem}[Theorem \ref{maintheorem}]
  Let $\Gamma$ be a discrete group satisfying the AP.
  Let $X$, $Y$ be $\Gamma$-spaces and $\pi$ be a proper factor map from $Y$ to $X$
  with the following condition$\colon$for every element $x$ in $X$ fixed by some non-neutral element of $\Gamma$, one has $|\pi^{-1}(x)| = 1$.
  Then the map
  \begin{align*}
    C_0(Z) \mapsto C_0(Z) \rtimes_r \Gamma
  \end{align*}
  gives a lattice isomorphism between the lattice of intermediate extensions of $\pi$ and that of intermediate \Cs-algebras of $C_0(X) \rtimes_r \Gamma \subset C_0(Y) \rtimes_r \Gamma$.
\end{theorem}

\subsection*{Notation}
Throughout this paper, $\Gamma$ denotes a discrete group.
The symbol `$\rtimes_r$' stands for the reduced \Cs-crossed product.
For $g \in \Gamma$, denote by $\lambda_g$ the unitary element of the reduced group \Cs-algebra ${\rm C}_r^*\Gamma$ correponding to $g$. 
For a unital $\Gamma$-\Cs-algebra $A$ and $g \in \Gamma$, denote $u_g$ the canonical implementing unitary element of $g$ in the reduced crossed product of $A \rtimes_r \Gamma$. 
For two elements $a$, $b$ of a \Cs-algebra, we denote $a \approx_\e b$ if $\|a-b\| < \e$.
We denote by $\N$ the set of positive integers.
Also, we denote by $\N_0$ the union of $\N$ and $\{0\}$. 
\subsection*{Acknowledgements}
The author would like to thank his supervisor, Professor Narutaka Ozawa for his support and many valuable comments.
He also thanks Professor Yuhei Suzuki for valuable comments and suggesting Example \ref{example:suzuki} and the arguments following it. 

\section{Preliminaries}
We say $X$ is a $\Gamma$-space if $X$ is a locally compact space equipped with a $\Gamma$-action by homeomorphisms. 

\begin{definition}
  For a $\Gamma$-space $X$ and an element $g$ of $\Gamma$, 
  we denote by $\Fix_X(g)$ the set of fixed points of $g$,
  i.e. $\{x \in X\mid g . x = x\}$.
  Also, we denote by $\mathcal{S}_X$ the set of all elements in $X$ which have non-trivial stabilizers. 
\end{definition}

\begin{definition}
  Let $X$, $Y$ be $\Gamma$-spaces. 
  A map $\pi$ from $Y$ to $X$ is said to be a factor map if it is a $\Gamma$-equivariant quotient map from $Y$ to $X$.
  We also refer to $\pi$ as an extension.
  A factor map is proper if the preimage of every compact set is compact.
\end{definition}

The approximation property (AP) has been introduced for locally compact groups by Haagerup--Kraus \cite{haagerup1994approximation}. 
In the discrete case, the AP is weaker than weak amenability and stronger than exactness.
See \cite{haagerup1994approximation} and Section 12 of \cite{brown2008c} for details.

Let $X$ be a $\Gamma$-space.
There exists a canonical conditional expectation $E$
from $C_0(X) \rtimes_r \Gamma $ onto $C_0(X)$
defined by 
\begin{align*}
  E(fu_g) =   
  \left\{
  \begin{array}{l}
    f \ \ \mbox{if} \, g = e\\
    0 \ \ \mbox{if} \, g \neq e, 
  \end{array}
  \right.
\end{align*}
for $f \in C_0(X)$ and $g \in \Gamma$.
Note that $E$ is faithful (see \cite{brown2008c}, Chapter 4.1). 

The following proposition plays an important role in the proof of the main theorem.
\begin{proposition}[\cite{suzuki2017group}, Proposition 3.4] \label{proposition:critical} 
  Let $\Gamma$ be a group satisfying the ${\rm AP}$.
  Let $A$ be a $\Gamma$-\Cs-algebra and let $X$ be a closed subspace of $A$.
  Let $x$ be an element of $A \rtimes_r \Gamma$ satisfying $E(xu_g^*) \in X$ for all $g \in \Gamma$.
  Then $x$ is contained in the closed subspace
  \begin{align*}
    X \rtimes_r \Gamma := \overline{\rm span}\{x u_g \mid x \in X, g \in \Gamma\}. 
  \end{align*}
\end{proposition}

\section{Main theorem}
In this section, we prove the Main Theorem (Theorem \ref{maintheorem}).
\begin{proposition}\label{proposition:bijection}
  Let $X$, $Y$ be $\Gamma$-spaces and $\pi$ be a $\Gamma$-equivariant quotient map from $Y$ to $X$.
  Assume for every element $x \in \mathcal{S}_X$, one has $|\pi^{-1}(x)| = 1$.
  Then for every subset $C$ of $\mathcal{S}_Y$, one has $\pi^{-1}\circ \pi(C) = C$.
\end{proposition}

\begin{proof}
  Let $\varphi \colon \mathcal{S}_Y \to \mathcal{S}_X$ be the restriction of $\pi$ to $\mathcal{S}_Y$.
  It suffices to show that $\varphi$ is bijective.
  For every $x \in \mathcal{S}_X$, there exists $e \neq g \in \Gamma$, $x = g.x$.
  Since $\pi$ is surjective, there exists $y \in Y$ satisfying $\pi(y) = x$.
  If $g.y \neq y$, then one has $|\pi^{-1}(x)| \ge 2$, which contradicts the assumption.
  Hence, $g.y = y$, that is $ y \in \mathcal{S}_Y$.
  Thus, $\varphi$ is surjective.
  Injectivity follows from the assumption.
\end{proof}

\begin{lemma}\label{lemma:important}
  Let $X$, $Y$ be  compact Hausdorff spaces and $\pi$ be a continuous map from $Y$ to $X$. 
  Let $K$ be a compact subset of $Y$ satisfying $K = \pi^{-1}\circ \pi(K)$ and $U$ be an open neighborhood of $K$.
  Then there exists an open neighborhood $V$ of $\pi(K)$ satisfying $\pi^{-1} (V) \subset U$. 
\end{lemma}

\begin{proof}
  Suppose there exists no such open neighborhood.
  Since $X$ is Hausdorff, there exists a decreasing net $\{V_i\}_{i \in I}$ of open sets such that $\bigcap_i V_i = \pi(K)$.
  By the assumption,  for every $i \in I$, one has $\pi^{-1}(V_i) \not\subset U$.
  Hence, there exist $x_i \in V_i$ and $y_i \in Y$ such that $\pi(y_i) = x_i$ and $y_i \not\in U$.
  Since $Y$ is compact, there exists a subnet $\{y_j\}_{j \in J}$ and $y \in Y$ such that $y_j \to y$.
  Since $\pi$ is continuous,
  one has $\pi(y_j) \to \pi(y)$.
  Then one has $\pi(y) \in \pi(K)$.
  Hence, one has $y \in K$.
  Since $U$ is an open neighborhood of $K$,
  there exists $j_0 \in J$ such that $y_{j_0} \in U$.
  This contradicts $y_j \not\in U$ for every $j \in J$.
\end{proof}

We say that $A\subset B$ is a non-degenerate inclusion of \Cs-algebras if every (or, equivalently, some) approximate unit of $A$ is an approximate unit of $B$. 
If $A\subset B$ is non-degenerate, the inclusion extends to an inclusion $M(A) \subset M(B)$ (see \cite[Proposition 2.1]{lance1995hilbert}).

We use the following lemma to show the non-compact case of Theorem \ref{maintheorem}.

\begin{lemma}\label{lemma:unitization}
  Let $A\subset B$ be a non-degenerate inclusion of \Cs-algebras.
  Let $C$ be a \Cs-subalgebra of the multiplier algebra $M(A)$.
  The map
  \begin{align*}
    \D \mapsto \D+C =: \D^\wedge
  \end{align*}
  defines a one-to-one correspondence between intermediate \Cs-algebras $A \subset \D \subset B$
  and intermediate \Cs-algebras $A+C \subset \D \subset B+C$, 
  with its inverse map given by
  \begin{align*}
    \D \mapsto \D \cap B =: \D^\vee.
  \end{align*}
\end{lemma}

\begin{proof}
  Let $\D$ be an intermediate \Cs-algebra of $A \subset B$.
  We first note that $C \subset M(\D)$ and
  $\D^\wedge = \D+C = Q_\D^{-1} \circ Q_\D(C) \subset M(\D)$ is a \Cs-subalgebra,
  where $Q_\D \colon M(\D) \to M(\D)/\D$ is the quotient map.
  Let $\{u_i\}$ be an approximate unit of $A$.
  By the non-degeneracy of the inclusion, $\{u_i\}$ is an approximate unit of $B$.

  Since $\D \subset B$, one has $\D \subset (\D^\wedge)^\vee $.
  Let $a$ be an element of $(\D^\wedge)^\vee$.
  Since $\D^\wedge = \D+C$, there exist $d \in \D$ and $c \in C$ such that $a = d+c$.
  Since $u_i c \in A$, one has $u_i(d+c) \in \D + A = \D$.
  Since $\{u_i\}$ is an approximate unit for $B$,
  one has $a \in \D$.
  Hence, one has $(\D^\wedge)^\vee = \D$.

  Let $\D$ be  an intermediate \Cs-algebra of $A+C \subset B+C$.
  Since $C \subset \D$, one has $(\D^\vee)^\wedge \subset \D$.
  Let $a$ be an element of $\D$.
  Since $a \in B+C$, there exist $b \in B$ and $c \in C$ such that $a = b+c$.
  Since $u_i b \in B$ and $u_i c \in A$, one has $u_i(b+c) \in B \cap \D$.
  Since $c - u_i c \in C+A$, one has $u_i b + c = u_i(b+c) + (c - u_i c) \in (\D \cap B) + C = (\D^\vee)^\wedge$.
  Hence, one has $(\D^\vee)^\wedge$.
\end{proof}

\begin{theorem}\label{maintheorem}
  Let $\Gamma$ be a discrete group satisfying the {\rm AP}.
  Let $X$, $Y$ be $\Gamma$-spaces and $\pi$ be a proper factor map from $Y$ to $X$
  such that $|\pi^{-1}(x)| = 1$ for every element $x$ of $\mathcal{S}_X$.
  We regard $C_0(X)$ as a $\Gamma$-\Cs-subalgebra of $C_0(Y)$ via $\pi$.
  Then the map
  \begin{align*}
    C_0(Z) \mapsto C_0(Z) \rtimes_r \Gamma
  \end{align*}
  gives a lattice isomorphism between the lattice of intermediate extensions of $\pi$ and that of intermediate \Cs-algebras of $C_0(X) \rtimes_r \Gamma \subset C_0(Y) \rtimes_r \Gamma$.
\end{theorem}

\begin{proof}
  We first show this theorem when $X$ and $Y$ are compact. 
  Let $E$ be the canonical conditional expectation from $C(Y) \rtimes_r \Gamma$ to $C(Y)$.
  Let $a$ be an element of $C(Y) \rtimes_r \Gamma$.
  Let  $\varepsilon > 0$ be given.
  There exist $n \in \N \cup \{0\}$, $t_k \in \Gamma$, and $f_k \in C(Y)$ $( k \in \{0, \ldots , n\})$
  such that $\| a - \sum_{k=0}^n f_k u_{t_k}\| < \e$, $t_0 = e$, and $t_k \neq e$ $( k \in \{1, \ldots, n\})$.
  
  Since $\Fix_X(t)$ is closed for every $t \in \Gamma$, for each $k \in \{1,2,\ldots, n\}$, there exists $\tilde{f_k} \in C(X)$ such that $f_k = \pi_\ast(\tilde{f_k})$ on $\Fix_Y(t_k)$ by the Tietze extension theorem and Lemma \ref{proposition:bijection}.

  For each $k \in \{1,\dots, n\}$, since $\Fix_Y(t_k)$ is compact, there exists an open set $W_k$ of $Y$ such that $\Fix_Y(t_k) \subset W_k$ and $|f_k - \pi_*(\tilde{f_k})| < \e/n$ on $W_k$.
  By Proposition \ref{proposition:bijection} and Lemma \ref{lemma:important}, there exists an open neighborhood $U_k^0$ of $\pi(\Fix_Y(t_k)) $ in $X$ such that $\pi^{-1}(U_k^0) \subset W_k$.
  
  For each $k \in \{1, 2, \ldots, n\}$, $t_k$ acts on $X \setminus U_k^0$ freely.
  Hence, there exist a finite subset $J_k' \subset \N$ and a finite open covering $\{U_k^j\}_{j \in J_k'}$ of $X \setminus U_k^0$ such that for every $j \in J_k'$,
  $t_k U_k^j \cap U_k^j = \emptyset$.
  Let $J_k = J_k' \cup \{0\} $ and $I = J_1 \times J_2 \times \cdots \times J_n$.
  For every $\i = (j_1, j_2,  \ldots, j_n) \in I$, we define $V_\i$ to be  $U_1^{j_1} \cap U_2^{j_2} \cap \cdots \cap U_n^{j_n}$. 
  We remark that for every $k \in \{1, 2, \ldots, n\}$ and every $\i = (j_1, j_2,  \ldots, j_n) \in I$ with $j_k \neq 0$, 
  one has $t_k U_\i \cap U_\i = \emptyset$.
  We also have $\bigcup_{ \substack{\i = (j_1, j_2,  \ldots, j_n)\\ j_k = 0}} U_\i= U_k^0$.
  For the open covering $\{V_\i\}_{\i \in I}$ of $X$, there exists $h_\i \in C(X)^+$ such that $\sum_{\i \in I} h_\i = 1$ and $\supp(h_\i) \subset V_\i$ $(\i \in I)$.
  We remark that for every $k \in \{1, 2, \ldots, n\}$ and every $\i = (j_1, j_2,  \ldots, j_n) \in I$ with $j_k \neq 0$,
  one has $h_\i(t_k.h_\i) = 0$.
  We define a c.c.p map $\Phi \colon C(Y) \rtimes_r \Gamma \to C(X) \rtimes_r \Gamma$ by the map $a \mapsto \sum_{\i \in I} h_\i^{\frac{1}{2}}a h_\i^{\frac{1}{2}}$.
  For each $k \in \{1, 2, \ldots, n\}$,
  we define $g_k := \Phi(u_{t_k}) u_{t_k}^*= \sum_{\i \in I}  h_\i^{\frac{1}{2}} (t_k. h_\i^{\frac{1}{2}}) \in C(X)$.
  For $ k \neq 0$, since $\supp(g_k) \subset U_k^0$ and $\|g_k\| \leq 1$, 
  one has $f_k g_k \approx_{\e/n} \tilde{f_k} g_k$.
  Therefore, 
  \begin{align*}
    \Phi(a) &\approx_\e \Phi( \sum_{k=0}^n f_k u_{t_k} )\\
    &= \sum_{k=0}^n f_k g_ku_{t_k} = f_0 +  \sum_{k=1}^n f_k g_ku_{t_k}\\
    &\approx_{(\e/n)\cdot n} f_0 +\sum_{k=1}^n \tilde{f_k} g_k u_{t_k} 
  \end{align*}
  Since $\Phi(a) \in C^*(a, C(X))$ and $\sum_{k=1}^n\tilde{f_k} g_k u_{t_k} \in C(X) \rtimes_r \Gamma$,
  one has
  \begin{align*}
    E(a) &\approx_\e f_0 \approx_{2\e} \Phi(a) - \sum_{k=1}^n \tilde{f_k} g_k u_{t_k} \in C^*(a, C(X) \rtimes_r \Gamma).
  \end{align*}
  Since $\e > 0$ was arbitrary,
  one has $E(a) \in  C^*(a, C(X) \rtimes_r \Gamma)$.
  
  Let $\D$ be an intermediate \Cs-algebra of $C(X) \rtimes_r \Gamma \subset C(Y) \rtimes_r \Gamma$.
  Then, by the result shown in the previous paragraph, one has $E(\D) \subset \D$.
  By Proposition \ref{proposition:critical}, for every intermediate \Cs-algebra $\D$ of $C(X) \rtimes_r \Gamma \subset C(Y) \rtimes_r \Gamma$,
  one has $\D = E(\D) \rtimes_r \Gamma$.
  
  Next, we show this theorem in the case where $X$ and $Y$ are noncompact.
  Let $\tilde{X} = X \sqcup \{x_\infty\}$, $\tilde{Y} = Y \sqcup \{y_\infty\}$ be the one-point compactifications of $X$, $Y$ respectively.
  Let $\tilde{\pi} \colon \tilde{Y} \to \tilde{X}$ denote the continuous extension of $\pi$.
  Since $\tilde{\pi}^{-1}(\{x_\infty\}) = \{y_\infty\}$, $\tilde{\pi}$ satisfies the assumption of this theorem.

  We will use a one-to-one correspondence between intermediate \Cs-algebras of $C_0(X)\rtimes_r \Gamma \subset C_0(Y) \rtimes_r \Gamma$ and
  that of $C( \tilde{X})\rtimes_r \Gamma \subset C(\tilde{Y}) \rtimes_r \Gamma$.
  Let $ \tilde{E}$ be the canonical conditional expectation from $C( \tilde{Y}) \rtimes_r \Gamma$ to $C(\tilde{Y})$.
  Let $\D$ be an intermediate \Cs-algebra of $C_0(X) \rtimes_r \Gamma \subset C_0(Y) \rtimes_r \Gamma$.
  Since $\D^\wedge := \D +  {\rm C}_r^*\Gamma$ is an intermediate \Cs-algebra of $C( \tilde{X}) \rtimes_r \Gamma \subset C( \tilde{Y}) \rtimes_r \Gamma$,
  one has $\D^\wedge = \tilde{E}( \D^\wedge) \rtimes_r \Gamma$.
  Let $(\D^\wedge)^\vee = \D^\wedge \cap (C_0(Y) \rtimes_r \Gamma)$. 
  By Lemma \ref{lemma:unitization} and Proposition \ref{proposition:critical}, one has
  \begin{align*}
    \D &= (\D^\wedge)^\vee\\
    &= ( \tilde{E}( \D^\wedge) \rtimes_r \Gamma) \cap (C_0(Y) \rtimes_r \Gamma) \\
    &= (\tilde{E}( \D^\wedge) \cap C_0(Y)) \rtimes_r \Gamma.
  \end{align*}
\end{proof}

\section{Examples}
\begin{example}[branched covering]
  Let $X$ be a complex plane $\C$.
  Let $k$ be an integer greater than or equal to $2$ and
  $Y = \{(z, w) \in \C \times \C \mid w^k = z\} $ be a Riemann surface of the $k$-th square root.
  Let $\pi$ be a projection from $Y$ to $X$,
  i.e., $Y \ni (z,w) \mapsto z \in X$.
  Let $\alpha \in \R\setminus \Q$.
  We define $\Z$-actions on $X$, $Y$
  by the following:
  for each $n \in \Z$, 
  \begin{align*}
    X \ni z &\mapsto ze^{2k\pi\alpha n} \in X \\
    Y \ni (z, w) &\mapsto (ze^{2k\pi\alpha n} , we^{2\pi\alpha n}) \in Y.
  \end{align*}
  Then $\pi$ is a $\Z$-equivariant proper quotient map and for every $n \in \Z$, one has $\Fix_Y(n) = \{(0,0)\}$ and $\Fix_X(n) = \{0\}$.
\end{example}

We show that the assumption of $\pi$ in Theorem \ref{maintheorem} is closed under taking the direct product.
It follows from the Whitehead lemma \cite{whitehead1948note} that $\pi_1 \times \pi_2$ in Lemma \ref{lemma:product} is quotient.
For the reader's convenience, we include the proof.
\begin{lemma}\label{lemma:invopen}
  Let $Y$ be a topological space and $X$ be a locally compact space.
  Let $\pi$ be a surjective proper continuous map and $U$ be a subset of $X$.
  If $\pi^{-1}(U)$ is open, $U$ is open. 
\end{lemma}

\begin{proof}
  Take an element $x$ in $U$.
  Let $V$ be a relative compact open neighborhood of $x$.
  It suffices to show $U \cap V$ is open. 
  Since $\pi^{-1}(U\cap V) = \pi^{-1}(U) \cap \pi^{-1}(V)$,
  $\pi^{-1}(U\cap V)$ is open.
  By continuity and properness,
  $\overline{V} \cap (U \cap V)^c = \pi( \pi^{-1}(\overline{V}) \cap \pi^{-1}(U \cap V)^c)$ is compact.
  We remark that
  \begin{align*}
    X = (U \cap V) \sqcup (( \overline{V} \cap( U \cap V)^c) \cup V^c). 
  \end{align*}
  Since $ \overline{V} \cap( U \cap V)^c$ and $V^c$ is closed, $(( \overline{V} \cap( U \cap V)^c) \cup V^c$ is closed.
  Then $U \cap V$ is open.
\end{proof}

\begin{lemma} \label{lemma:product}
  For each $i \in \{1, 2\}$, let $X_i$, $Y_i$ be $\Gamma$-spaces
  and $\pi\colon Y_i \to X_i$ be a proper factor maps.
  We see $X_1 \times X_2$ and $Y_1 \times Y_2$ are $\Gamma$-spaces with diagonal actions. 
  Then $\pi := \pi_1 \times \pi_2\colon  Y_1 \times Y_2 \to X_1 \times X_2$ is a proper factor map.  
\end{lemma}

\begin{proof}
  It suffices to check properness and quotientness.
  We will show properness.
  Let $C \subset Y_1 \times Y_2$ be a compact subset.
  For each $i \in \{1,2\}$, let $p_i\colon  Y_1 \times Y_2 \to Y_i$ be $i$-th projections.
  By the $\pi_i$'s continuity, for each $i \in \{1,2\}$, $p_i(C)$ is compact.
  Since for each $i$, $\pi_i^{-1}(p_i(C))$ is compact,
  $\pi^{-1}(C) \subset \pi_1^{-1}(p_1(C)) \times \pi_2^{-1}(p_2(C))$ is compact.

  We will show quotientness.
  Let $U \subset X$.
  Since $\pi$ is continuous, if $U$ is open, then $\pi^{-1}(U)$ is open.
  By Lemma \ref{lemma:invopen}, if $\pi^{-1}(U)$ is open, then $U$ is open.
\end{proof}

\begin{example}
  For each $i \in \{1,2\}$, let $X_i$, $Y_i$ be $\Gamma$-spaces and $\pi \colon Y_i \to X_i$ be a $\Gamma$-equivariant proper factor map.
  By the above lemma,
  if for every $i \in \{1,2\}$, $\pi_i$ is a proper factor map such that $|\pi_i^{-1}(x)| = 1$ for every element $x$ in $\mathcal{S}_{X_i}$, 
  then
  so is $\pi_1 \times \pi_2 \colon Y_1 \times Y_2 \to X_1 \times X_2$.
\end{example}

We show that some compactifications satisfy the assumption in the main theorem.
For a locally compact space $X$, we denote by $\beta X$ the Stone-$\check{{\rm C}}$ech compactification of $X$. 
The following proposition is easily seen from the universal property of the Stone-$\check{{\rm C}}$ech compactification and Lemma \ref{lemma:invopen}. 
\begin{proposition}
  Let $\varphi \colon Y \to X$ be a proper quotient map between locally compact spaces $X$ and $Y$.
  Let $\tilde{X} = X \cup \d X$ be a compactification of $X$.
  Then there exists a quotient map $\beta \varphi \colon \beta Y \to \tilde{X}$ extending $\varphi$.
\end{proposition}

\begin{remark}
  By the above proposition, one has
  \begin{align*}
    (\beta \varphi) (\beta Y \setminus Y) = \d X \text{ and } \beta \varphi (Y) = X.
  \end{align*}
  It follows that $(\beta \varphi)_* (C( \tilde{X})) \cap C_0(Y) = (\beta \varphi)_*(C_0(X))$.
\end{remark}

\begin{proposition}\label{compactification}
  With the same assumptions as the previous proposition,
  we define $\tilde{Y}$ to be the character space of the \Cs-subalgebra of $C(\beta Y)$ generated by $(\beta \varphi)_*(C(\tilde{X}))$ and $C_0(Y)$.
  Then one has
  \begin{align*}
    C(\tilde{Y})/C_0(Y) \cong \beta \varphi_* (C ( \d X)).
  \end{align*}
\end{proposition}

\begin{proof}
  \begin{align*}
    C(\tilde{Y})/C_0(Y) &= ((\beta \varphi)_*(C(\tilde{X})) + C_0(Y))/C_0(Y)\\
    &\cong (\beta \varphi)_*(C(\tilde{X}))/(\beta \varphi)_*(C(\tilde{X})) \cap C_0(Y)\\
    &= (\beta \varphi)_*(C(\tilde{X}))/\varphi_*(C_0(X))\\
    &\cong \beta \varphi_* (C ( \d X)).
  \end{align*}
\end{proof}

\begin{remark}
  Let $\d Y$ to be $\tilde{Y} \setminus Y$. 
  Then one has
  \begin{align*}
    C(\d Y) \cong \beta \varphi_* (C ( \d X)) \text{ and } \beta \varphi(\d Y) \cong \d X.
  \end{align*}
\end{remark}

For a compact space $X$ and an open subset $U$ of $X$, 
we consider the \Cs-subalgebra of $\ell_\infty(X)$ generated by $C(X)$ and the characteristic function $\chi_{\overline{U}}$.

\begin{proposition}
  Let $X$ be a compact space and $U$ be an open subset of $X$.
  Let $p = \chi_{\overline{U}} \in l_\infty(X)$,
  Then one has $ l_\infty(X) \supset \Cs(C(X), p) \cong C( \overline{U}) \oplus C(X \setminus U)$
\end{proposition}

\begin{proof}
  Let $V$ be a subset of $X$.
  Let $q = \chi_V$ be the characteristic function in  $\ell_\infty(X)$.
  We will show $C(X)q \cong C( \overline{V})$.
  We define $\varphi \colon C(X)q \to C( \overline{V})$ by $fq \mapsto f |_{\overline{V}}$ for $f \in C(X)$.
  The well-definedness and the injectivity of $\varphi$ follows from continuity.
  The surjectivity of $\varphi$ follows from the Tietze's extension theorem.
  Hence, $\varphi$ is an isomorphism.
  
  Hence, one has $\Cs(C(X), p) \cong C(X)p \oplus C(X)(1-p) \cong C( \overline{U}) \oplus C(X \setminus U)$.
\end{proof}

\begin{remark}
  Under the above assumption,
  let $\tilde{X}$ to be the character space of $\Cs(C(X), p)$ and $ \iota \colon C(X) \to C(\tilde{X})$ be the inclusion of $C(X)$ to $C(\tilde{X})$.
  By the above proposition,
  one has $\tilde{X} \cong \overline{U} \sqcup (X\setminus U)$.
  Hence, $\pi := \iota_* \colon \tilde{X} \to X$ is a proper quotient map,
  and for every $x \in X \setminus \d U$, one has $|\pi^{-1}(x)| = 1$.
  Furthermore, for every $x \in X$,
  $x$ belong to $\d U$ if and only if one has $|\pi^{-1}(x)| = 2$.
\end{remark}

\begin{example}
  Let $X$ be a compact $\Gamma$-space,  $U$ be an open subset of $X$ and
  $p = \chi_{\overline{U}}\in l_\infty(X)$.
  Let compact $\Gamma$-space $\tilde{X}$ to be the character space of $\Cs(C(X), \Gamma \cdot p)$
  , $ \iota \colon C(X) \to C(\tilde{X})$ be the inclusion of $C(X)$ to $C(\tilde{X})$
  and $\pi := \iota_* \colon \tilde{X} \to X$.
  Then $\pi$ is a proper factor map.
  If $\mathcal{S}_X \cap \Gamma \cdot \d U = \emptyset$,
  then 
  for every $x \in \mathcal{S}_X$,
  one has $|\pi^{-1}(x)| = 1$. 
\end{example}

\begin{proof}
  We prove $|\pi^{-1}(x)| = 1$ for every $x \in \mathcal{S}_X$.
  It suffices to show that for every $ x \in X$, one has $x \in \Gamma \cdot \d U$ if and only if $|\pi^{-1}(x)| \geq 2$.
  
  Let $x \in \Gamma \cdot \d U$.
  There exists $g \in \Gamma$ s.t. $x \in g.\d U$.
  Let $X_1$ to be the character space of $\Cs(C(X), g.p)$. 
  Let $\iota_1$ be the inclusion of $C(X)$ to $\Cs(C(X), g.p)$ and $\iota_2$ be the inclusion of $\Cs(C(X), g.p)$ to $C( \tilde{X})$.
  Let $\pi_1 = (\iota_1)_* \colon X_1 \to X$ and $\pi_2 = ( \iota_2)_*\colon \tilde{X} \to X_1$.
  Since $\pi^{-1}(x) = \pi_2^{-1} (\pi_1^{-1}(x))$,
  one has $|\pi^{-1}(x)| = |\pi_2^{-1} (\pi_1^{-1}(x))| \geq |\pi_1^{-1}(x)| = 2$ by the above remark.
  
  Let $x$ be an element of $X$ such that $|\pi^{-1}(x)| \geq 2$.
  There exist distinct elements $y_1$ and $y_2$ in $\tilde{X}$ such that $\pi(y_1) = \pi(y_2) = x$.
  Let $\Gamma = \{g_n\}_{n \in \N}$.
  Let $A_n$ be a \Cs-subalgebra of $C( \tilde{X})$ generated by $C(X)$ and $\{g_k.p\}_{k=1}^n$.
  Then $\cup_{n} A_n$ is dense in $C( \tilde{X})$.
  So, there exist $n \in \N$ and $a \in A_n$ such that $y_1(a) \neq y_2(a)$.
  We regard $y_1,\, y_2$ as characters on $A_n$.
  By the above remark, there exists $g \in \Gamma$ such that $x \in g.\d U$.  
\end{proof}

The assumption in the above remark is satisfied in many cases including the following case.
\begin{example}
  Let $X$ be a compact metric $\Gamma$-space.
  Assume $\mathcal{S}_X$ is separable and the Lebesgue covering dimension of $\mathcal{S}_X$ is zero. 
  Let $A$, $B$ be disjoint closed subsets in $X$. 
  Then, by the second separation theorem (see \cite[1.5.13]{engelking1978dimension}),
  there exists an open set $U$ in $X$ s.t. $A \subset U$, $B \subset U^c$ and $\d U \cap \mathcal{S}_X = \emptyset$.
  Since $\mathcal{S}_X$ is $\Gamma$-invariant, one has $\mathcal{S}_X \cap \Gamma \cdot \d U = \emptyset$.
\end{example}

We construct another example.
We are grateful to Yuhei Suzuki for letting us know the following example.
We say an element $g$ in a free group is indivisible if it is not a proper power of some element in the free group.
For every element $x$ in a free group, we denote by $x^\infty$ the limit of $x^n$ in the Gromov compactification of the free group. 

\begin{remark}\label{remark:indivisible}
  We consider the free group $\mathbb{F}_d$ of rank $d$, where $1 < d < \infty$, and its action on its Gromov boundary $\d \mathbb{F}_d$.
  Let $x$ be an element of $\mathbb{F}_d$.
  The stabilizer group of $x^\infty$ is a cyclic group and we denote by $y$ its generator.
  Then $y$ is indivisible.
  Also, one has either $x^\infty = y^\infty$ or $x^\infty = y^{-\infty}$.
  By replacing $y$ with $y^{-1}$, we can choose an element $y$ satisfying $x^\infty = y^\infty$.
  Also, for an indivisible element $x$, every element in $\mathbb{F}_d$ which fixes $x^\infty$ is represented by the power of $x$. 
\end{remark}

For a subgroup $\Lambda$ of $\Gamma$ and an element $t$ of $\Gamma$,
we denote by $C_\Lambda(t)$ the $\Lambda$-conjugacy class of $t$. 

\begin{example}\label{example:suzuki}
  Let $\mathbb{F}_d := \langle a_1, \ldots, a_d\rangle$ be the free group of rank $d$, where $1 < d < \infty$. 
  Let $\d \mathbb{F}_d$ denote the Gromov boundary.
  Let $\Gamma$ be a non-trivial normal subgroup of $\mathbb{F}_d$.
  Then the $\Gamma$-action on $\d \mathbb{F}_d$ is topologically free and minimal.
  Let $T$ be a subset of $\mathbb{F}_d$ which consists of indivisible elements
  such that $|\{C_{\F_d}(t) \mid t \in T\}| < \infty$. 
  Assume that for every non-zero $n \in \Z$ and for every $t \in T$, $t^n \notin \Gamma$.
  Let
  \begin{align*}
    \mathcal{R}_{T, \Gamma} := \{(x,x) \mid x \in \d \mathbb{F}_d\}\cup\{(g.t^\infty, g.t^{-\infty})\mid g \in \Gamma,\, t \in T \cup T^{-1} \}.
  \end{align*}
  Then $\mathcal{R}_{T, \Gamma}$ is a $\Gamma$-invariant equivalence relation on $\d \mathbb{F}_d$.
  Also, $\d \mathbb{F}_d/ \mathcal{R}_{T, \Gamma}$ is Hausdorff.
  
  Furthermore, 
  the quotient map  $\pi \colon \d \mathbb{F}_d \to \d\mathbb{F}_d / \mathcal{R}_{T, \Gamma}$ is
  a proper quotient map satisfying for every $x \in \d\mathbb{F}_d / \mathcal{R}_{T, \Gamma}$ with a non-trivial stabilizer, $|\pi^{-1}(x)|=1$.
\end{example}

\begin{proof}
  Since the $\F_d$-action on $\d \F_d$ is topologically free,
  the $\Gamma$-action on $\d \F_d$ is topologically free.
  We will show the minimality of the $\Gamma$-action.
  Let $S := \{a_1, \ldots, a_d\}$, which is a finite generating set of $\d \mathbb{F}_d$. 
  We regard $\d \mathbb{F}_d$ as the set of infinite reduced words of $\mathbb{F}_d$ (see \cite[5.1]{brown2008c}). 
  Let $x = x_1x_2 \cdots$ and $y = y_1 y_2 \cdots$ be elements in $\d \mathbb{F}_d$,
  where $x_i, y_i \in S \cup S^{-1}$.
  Let $\gamma \in \Gamma$ be a non-trivial element. 
  For each $n \in \N$, there exists $z_n \in S \cup S^{-1}$ such that $|y_n z_n| > |y_n|$ and $|z_n\gamma| > |\gamma|$,
  where $|\cdot|$ is the length function on $\Gamma$ determined by $S$.
  Let $w_n = (y_1\cdots y_n)z_n\gamma z_n^{-1}(y_1 \cdots y_n)^{-1} \in \mathbb{F}_d$.
  Then one has $w_nx \to y$.
  Also, by the normality of $\Gamma$, $w_n$ belongs to $\Gamma$.
  Hence, the $\Gamma$-action on $\d \mathbb{F}_d$ is minimal.

  Since $|\{C_{\F_d}(t) \mid t \in T\}| < \infty$,
  there exists a finite subset $T'$ of $T$ such that $\{C_{\F_d}(t) \mid t \in T\} = \{C_{\F_d}(t) \mid t \in T'\}$.
  So, one has $\mathcal{R}_{T,\Gamma} \subset \mathcal{R}_{T', \F_d}$. 
  Hence, in the similar way as \cite[Lemma 4.4]{suzuki2017group},
  we can show that $\mathcal{R}_{T, \Gamma}$ is a $\Gamma$-invariant equivalence relation on $\d \mathbb{F}_d$ and $\d \mathbb{F}_d/ \mathcal{R}_{T, \Gamma}$ is Hausdorff.

  We will show that $\pi$ satisfies that for every $x \in \d\mathbb{F}_d / \mathcal{R}_{T, \Gamma}$ with a non-trivial stabilizer, one has $|\pi^{-1}(x)|=1$.
  Let $x \in \d\mathbb{F}_d / \mathcal{R}_{T, \Gamma}$ satisfying $\gamma.x = x$ for some non-neutral element $\gamma \in \Gamma$.
  We remark each equivalence class of $\mathcal{R}_{T, \Gamma}$ contains at most two elements. 
  Suppose there exist $g \in \Gamma$ and $t \in T \cup T^{-1}$ such that $\pi^{-1}(\{x\}) = \{g.t^\infty, g.t^{-\infty}\}$.
  Then one has $\gamma. g. t^\infty \in \{g.t^\infty, g.t^{-\infty}\}$. 
  Since there exists no element $h \in \mathbb{F}_d$ with $h.t^\infty = t^{-\infty}$,
  we may assume $\gamma. g. t^\infty = g.t^\infty$. 
  Since $g^{-1}\gamma g$ fixes $t^\infty$, 
  there exist a non-zero integer $n$ such that $g^{-1}\gamma g = t^n$ by Remark \ref{remark:indivisible}.
  This contradicts the assumption.
  Hence, one has $|\pi^{-1}(x)| = 1$.
\end{proof}

By the following proposition and Theorem \ref{maintheorem},
one has a one-to-one correspondence between intermediate \Cs-algebras of $C(\d\mathbb{F}_d / \mathcal{R}_{T, \Gamma}) \rtimes_r \Gamma \subset C( \d\mathbb{F}_d) \rtimes_r \Gamma$ and subsets of $T$.
\begin{proposition} \label{proposition:conjugacy}
  In addition to the above condition,
  assume for distinct elements $s$, $t \in T$, one has $C_\Gamma(s) \cap \{t, t^{-1}\} = \emptyset$. 
  Then the map
  \begin{align*}
    ( \rho \colon \d \mathbb{F}_d \to Z) \mapsto S_\rho , 
  \end{align*}
  where $S_\rho := \{ t \in T \mid |\rho^{-1}\circ \rho(t)| = 2\}$,
  gives a one-to-one correspondence between intermediate extensions of $\pi$ and
  subsets of $T$ 
  with its inverse map given by
  \begin{align*}
    S \mapsto ( \rho_{S} \colon \d \mathbb{F}_d \to \d\mathbb{F}_d / \mathcal{R}_{S, \Gamma}),
  \end{align*}
  where $\rho_{S} \colon \d \mathbb{F}_d \to \d\mathbb{F}_d / \mathcal{R}_{S, \Gamma}$ is the quotient map.
\end{proposition}

\begin{proof}
  Let ($\rho$, $\rho'$) be an intermediate extension of $\pi$, where $\rho \colon \d \mathbb{F}_d \to Z$ and $\rho' \colon Z \to \d \mathbb{F}_d / \mathcal{R}_{T, \Gamma}$ are factor maps such that $\rho' \circ \rho = \pi$.
  Since $\rho_{S_\rho}$ is surjective, 
  we define the $\Gamma$-equivariant map $f \colon \d \mathbb{F}_d / \mathcal{R}_{S_\rho, \Gamma} \to Z$ by $x \mapsto \rho(y)$ for some $y \in \d \mathbb{F}_d$ with $\rho_{S_\rho}(y) = x$.
  We first check this map is well-defined.
  Let $x \in \d \mathbb{F}_d/\mathcal{R}_{S_\rho, \Gamma}$.
  Let $y_1$ and $y_2$ be elements in $\d \mathbb{F}_d$ such that $\rho_{S_\rho}(y_1) = \rho_{S_\rho}(y_2) = x$.
  We may assume $y_1 \neq y_2$. 
  By the definition of $\mathcal{R}_{S_\rho, \Gamma}$, we may assume $y_1 = g.s^\infty$ and $y_2 = g.s^{-\infty}$ for some $g \in \Gamma$ and some $s \in S_\rho$.
  By the definition of $\mathcal{R}_{S_\rho, \Gamma}$,
  one has $(g.s^\infty, g.s^{-\infty}) \in \mathcal{R}_{S_\rho, \Gamma}$.
  Hence, one has $\rho(y_1) = \rho(y_2)$.
  Hence, $f$ is well-defined.
  Similarly, we can check $f$ is $\Gamma$-equivariant. 

  We will show $f$ is a continuous bijective map, that is an isomorphism. 
  By the definition of $f$, one has $\rho = f \circ \rho_{S_\rho}$.
  Hence, $f$ is surjective.
  Also, since $\rho$ and $\rho_{S_\rho}$ are quotient maps, $f$ is continuous.
  We will show the injectivity of $f$.
  Let $x_1$, $x_2$ be elements in $\d \mathbb{F}_d/\mathcal{R}_{S_\rho, \Gamma}$ with $f(x_1) = f(x_2)$.
  There exist $y_1$ and $y_2$ in $\d \mathbb{F}_d$ such that $\rho_{S_\rho}(y_i) = x_i$ for each $i \in \{1,2\}$.
  Then one has $\rho(y_1) = \rho(y_2)$. 
  We may assume $y_1 \neq y_2$.
  By the definition of $S_\rho$, one has $(y_1, y_2) \in \mathcal{R}_{S_\rho, \Gamma}$. 
  So, one has $x_1 = x_2$.
  Hence, $f$ is injective.
  Hence, $\d \mathbb{F}_d/\mathcal{R}_{S_\rho, \Gamma} \cong Z$ as an intermediate extension of $\pi$.

  Let $S$ be a subset of $T$.
  We will show that $S = S_{\rho_S}$. 
  Let $s \in S$.
  Since $(s^\infty, s^{-\infty}) \in \mathcal{R}_{S, \Gamma}$,
  one has $s \in S_{\rho_S}$.
  Let $s \in S_{\rho_S}$.
  Then one has $(s^\infty, s^{-\infty}) \in \mathcal{R}_{S, \Gamma}$.
  So, there exist $g \in \Gamma$ and $s' \in S$ such that $s'^\infty = g.s^\infty$ or $s'^\infty = g.s^{-\infty}$.
  Suppose $s \neq s'$. 
  Since $gsg^{-1}$ fixes $g.s^\infty$ and $g.s^{-\infty}$, and $s'$ is indivisible,
  then there exist $n \in \Z \setminus \{0\}$ such that $s'^n = gsg^{-1}$ by Remark \ref{remark:indivisible}.
  Since $s$ is indivisible,
  one has $n= \pm 1$.
  This contradicts the assumption, that is $C_\Gamma(s) \cap \{s', s'^{-1}\} = \emptyset$,
  where $[s]$ is the $\Gamma$-conjugacy class of $s$.
  Hence, one has $s = s' \in S$. 
\end{proof}

\begin{cororally}
  In the above conditions, 
  the map
  \begin{align*}
    S \mapsto C( \d \mathbb{F}_d / \mathcal{R}_{S, \Gamma}) \rtimes_r \Gamma
  \end{align*}
  gives a lattice isomorphism between the lattice of subsets of $T$ and that of intermediate \Cs-algebras of $C( \d \mathbb{F}_d / \mathcal{R}_{T, \Gamma}) \rtimes_r \Gamma \subset C( \d \mathbb{F}_d) \rtimes_r \Gamma$.
\end{cororally}

In many cases, we can construct an infinite subset $T$ of $\F_d$ satisfying the above conditions as follows. 
\begin{proposition}
  Let $\mathbb{F}_d$ be the free group of rank $d$ $(d > 1)$.
  Let $\Gamma$ be a normal subgroup of $\F_d$ such that the  quotient group $\F_d / \Gamma$ is not virtually cyclic.
  Then
  there exists an infinite subset $T$ satisfying the following properties:
  \begin{enumerate}
  \item Every element of $T$ is indivisible.
  \item For every $0 \neq n \in \Z$ and $t \in T$, one has $t^n \notin \Gamma$.
  \item For every distinct elements $s$, $t \in T$, one has $C_\Gamma(s) \cap \{t,t^{-1}\} = \emptyset$.
  \item $|\{C_{\F_d}(t) \mid t \in T\}| < \infty$.
  \item $|\{C_\Gamma(t) \mid t \in T\}| = \infty$.
  \end{enumerate}
\end{proposition}

\begin{proof}
  Since $\F_d/\Gamma$ is not virtually cyclic,
  there exists an indivisible element $t$ of $\F_d$ such that
  for every $0 \neq n \in \Z$, $t^n \notin \Gamma$.
  We denote by $\langle t\rangle$ the subgroup of $\F_d$ generated by $t$.
  Let $S$ be a subset of $\F_d$ such that $\bigsqcup_{s \in S} C_\Gamma(sts^{-1}) = C_{\F_d}(t)$. 
  We show $|S| = [\F_d / \Gamma \colon \langle t\rangle \Gamma]$,
  where $ [\F_d / \Gamma \colon \langle t\rangle \Gamma]$ is the index of $\langle t\rangle \Gamma$ in $\F_d / \Gamma$.
  Let $s$, $s' \in \F_d$.
  Let $C_\Gamma(sts^{-1}) = C_\Gamma(s'ts'^{-1})$.
  Then there exists $g \in \Gamma$ such that $gsts^{-1}g^{-1} = s'ts'^{-1}$.
  Hence, one has $s'^{-1}gs t = t s'^{-1}gs$.
  Since the centralizer of $t$ is $\langle t\rangle $,
  one has $s'^{-1}gs \in \langle t\rangle $. 
  Hence, one has $s \Gamma = s' \langle t\rangle\Gamma$.
  Hence, one has $|S| = [\F_d / \Gamma \colon \langle t\rangle\Gamma]$.
  Since $\F_d / \Gamma$ is not virtually cyclic,
  one has $ [\F_d / \Gamma \colon \langle t\rangle \Gamma] = \infty$.
  Hence, one has $|S| = \infty$.
  
  Let $T = \{sts^{-1} \mid s \in S\}$. 
  By construction, $T$ satisfies the properties (4) and (5).
  Since $\Gamma$ is normal and $t$ is indivisible,
  $T$ satisfies the properties (1) and (2).
  We show $T$ satisfies the property (3).
  Let distinct elements $s$, $t \in S$.
  By construction, one has $ t \notin C_\Gamma(s)$.
  Suppose there exists $g \in \Gamma$ such that $t^{-1} = gst(gs)^{-1}$.
  Then one has $t^{-\infty} = gs.t^\infty$,
  which is a contradiction, 
  since there exists no element $h$ of $\F_d$ such that $h.t^\infty = t^{-\infty}$,
\end{proof}

\end{document}